\documentclass[final,12pt]{article}
\usepackage{hyperref}
\usepackage{amssymb,amsmath,amsthm}
\usepackage{cite}
\usepackage{tikz}
\usetikzlibrary{arrows,backgrounds,decorations.pathmorphing,decorations.pathreplacing,positioning,fit}
\usepackage{enumerate}
\usepackage[conditional,light,first,bottomafter]{draftcopy}
\draftcopyName{DRAFT\space\today}{130}
\draftcopySetScale{65}
\usepackage[letterpaper,hmargin=3.7cm,vmargin=3.3cm]{geometry}
\usepackage{setspace}
\makeatletter
\renewcommand{\section}{\@startsection%
{section}%
{1}%
{0em}%
{1.7em}%
{1.2em}%
{\normalfont\large\centering\bfseries}}
\renewcommand{\@seccntformat}[1]%
{\csname the#1\endcsname.\hspace{0.5em}}
\makeatother

\numberwithin{equation}{section}
\newtheorem{teorema}{Theorem}[section]
\newtheorem{proposicion}{Proposition}

\theoremstyle{definition}

\begin{document}
\begin{titlepage}
\title{A short proof of F. Riesz representation Theorem
}

\author{
\textbf{Rafael del Rio}
\\
\small \texttt{delrio@iimas.unam.mx}
\\
\textbf{Asaf L. Franco}
\\
\small \texttt{asaflevif@hotmail.com}
\\
  \textbf{Jose A. Lara}   
\\
\small\texttt{nekrotzar.ligeti@gmail.com}
\\
\small Departamento de F\'{i}sica Matem\'{a}tica\\[-1.6mm]
\small Instituto de Investigaciones en Matem\'aticas Aplicadas y en Sistemas\\[-1.6mm]
\small Universidad Nacional Aut\'onoma de M\'exico\\[-1.6mm]
\small C.P. 04510, M\'exico D.F.\\[-1.6mm]
\\[2mm]
}
\date{}
\maketitle
\vspace{-4mm}
\begin{center}
\begin{minipage}{5in}
  \centerline{{\bf Abstract}} \bigskip
A direct proof  of the  Riesz representation theorem is provided. This theorem characterizes the linear functionals acting  on the vector space $\mathcal C(K)$ of continuous functions  defined on a compact subset $K$ of the real numbers $\mathbb{R}$. The proof avoids complicated arguments commonly used in generalizations of Riesz original theorem.

  \end{minipage}
\end{center}
\thispagestyle{empty}
\end{titlepage}
\section{Introduction}
\label{sec:intro}
The Riesz representation theorem is a remarkable result which describes the continuous linear functionals acting on the space of continuous functions defined on a set $K$. It is very surprising that all these functionals are just  integrals and vice versa. In case  $K$ is a closed interval of real numbers, any such functional is represented by  Riemann-Stieltjes integral,  which is a generalization of the usual Riemann integral. This  was  first announced by F. Riesz   in 1909  \cite{Riesz}.  In case $K$ is  compact set (not necessarily  a closed interval), then a more general concept of integral is needed, because the Riemann-Stieltjes integral  used by Riesz is defined only for functions on intervals. In this work, we prove that there is a short path between the two cases.


 Besides its aesthetic appeal, the above mentioned theorem has far-reaching applications.
 It allows a short  proof of the Kolmogoroff consistency theorem, see \cite{Co} thm 10.6.2., and can be used to give an elegant proof of the spectral theorem for selfadjoint bounded operators, see section VII.2 of \cite{MR751959}.
 Both these theorems are main results in probability and functional analysis respectively. Moreover,  the entire theory of integration for general spaces can be recovered using  the theorem of Riesz. See for example \cite{Rud}, where the Lebesgue measure on $\mathbb R^n$ is constructed. More  generally it can also be used to show the existence of the Haar measure on a group, see  \cite{Co} chap. 9. 
 
%
%
 In this note we  give a short proof of the Riesz representation theorem for the case $K$ is an arbitrary compact set of real numbers, see Theorem \ref{RC} below. This is interesting because in many situations we have a compact set which is not a closed interval. To prove the spectral theorem, for example, one considers the set of continuous functions defined on the spectrum of selfadjoint bounded operator, which is a compact set of $\mathbb R$, but not necessarily a closed interval.  We get our result  starting from the nondecreasing function that appears in the Riemann-Stieltjes integral representation of Riesz original formulation. To this function we associate a measure which is used to integrate over  general compact sets. Then we show how  this  Lebesgue integral representation can be seen as a Riemann-Stieltjes integral again.  Our proof is new, its reliance on measure theory make it not completely elementary, but it is very direct and quite simple.

 \section{Preliminaries } \
  
  Let us introduce first some definitions and notations we shall use.
   \subsection {Definitions and notation.}
   
  Let $\mathcal{C}(K):=\{ f:K\to \mathbb{R} : f \text{ continuous}\} $ where $K$ is a compact subset of  $\mathbb{R}$, the real numbers.
 A    {\it functional}  is an assignment  $L: \mathcal{C}(K)\to \mathbb{R}$. The functional is linear if   
 $L(c_1f+c_2g)=c_1L(f)+c_2 L(g)$ for all $f,g\in \mathcal{C}(K), c_1,c_2 \in \mathbb{R} $. It is continuous  if there exists a fixed $M>0$ such that $ |Lf|\le M \|f\|_\infty$ for all $f\in \mathcal{C}(K)$, where $\| \cdot\|_\infty$ denotes the uniform norm, that is, $\|f\|_\infty=\sup \{|f(x)|: x\in K\}$.
 We define the norm of such functional  as 
$$\|L\|_{\mathcal{C}(K)}=\|L\|= \sup \{|L(f)|:f\in \mathcal{C}(K) \text{  and  }\,\|f\|_\infty \le 1 \},$$

 We denote the set of the {\it linear continuous functionals} on  $\mathcal{C}(K)$ by $\mathcal{C}(K)^*$. It is called the {\it dual space}. In general, the dual of normed linear space $X$ is denoted by $X^*$.
  A functional $L$ on $\mathcal{C}(K)$ is said to be a {\it positive}  if $L(f)\geq 0$ whenever $f(x)\geq 0$ for 
   every $x\in\mathbb{R}$. We use the notation  $C(K)^*_+$  for the set of { \it positive linear functionals} on $C(K)$.\
   
The function  $\alpha:[a,b]\longrightarrow\mathbb R$ is said to be {\it normalized}, if $\alpha(a)=0$ and $\alpha(t)=\alpha(t+),\, $  $a< t< b$, that is, $\alpha$ is continuous from the right inside the interval (not at a!  If it were right continuous at $a$, theorem (\ref{R}) would not hold for the  functional $L(f)=f(a)$). The total variation of  a monotone increasing function $\alpha$  is defined as $V(\alpha)= \alpha (b)-\alpha (a)$.
We denote  the characteristic function of a set $A\subset [a,b]$  by $ \mathbf{1}_{A}$ where  
$\mathbf{1}_{A}(x)=1$ if $x \in A$ and $0$ if $x \in [a,b]\setminus A$.
 
 \subsection{Representation theorem for functionals on $\mathcal C[a,b]$.}\
 
 We formulate the above-mentioned result by F. Riesz as follows:

 
\begin{teorema} \label{R}
Let $L:\mathcal{C}[a,b]\longrightarrow\mathbb R$ be a positive linear functional. There exists a unique normalized monotone function $\alpha:[a,b]\longrightarrow\mathbb R$ such that
\begin{equation}\label{original}
Lf=\int_a^bf(x)d\alpha(x).
\end{equation}
The integral is understood in the sense of Riemann-Stieltjes.
 Moreover $\| L \| = V(\alpha).$

\end{teorema}

The Riemann-Stieltjes integral is a generalization of the Riemann integral, where instead of taking the length of the intervals,  a $\alpha$-weighted length is taken. For an interval $I$ the $\alpha$-length is given by $\alpha(I)= \alpha(y)-\alpha (x)$, where $x,y$ are the end points of $I$ and $\alpha$ is a function of finite variation. The integral of a continuous function $f$ on $[a, b]$ is defined as the limit, when it exists, of the sum $\sum_i f(c_i) \alpha (I_i)$ where $\{I_i\}$ is a finite collection of subintervals whose endpoints form a partition of $[a,b]$ and $c_i\in I_i$. See \cite{MR1068530} p.122.

There are different proofs of the above theorem, see  for example \cite{MR1861991}. Here we will give a sketch of the proof which uses  the following result  about extensions of functionals known as the Hahn-Banach theorem:

 {\it Let $X$ a normed linear space, $Y$ a subspace of $X$, and $\lambda$ an element of $Y^*$. Then there exists a $\Lambda\in X^*$ extending $\lambda$ with the same norm}. See \cite{MR751959} for a proof.

\begin{proof}[ Proof of theorem \ref{R}].\
 We may assume that $[a,b]=[0,1]$. Since $L\in \mathcal{C}[0,1]^*$ we use Hahn-Banach theorem to conclude the existence of $\Lambda \in B[0,1]^*$ such that $\|\Lambda\|=\|L\|$ and $L=\Lambda$ on $\mathcal{C}[0,1]$ and where $B[0,1]$ is the set of bounded functions on $[0,1]$. \\

Let us define the functions $\mathbf{1}_x:= \mathbf{1}_{[0,x]}$, that is $\mathbf{1}_x(t)=1$ when $t\in [0,x]$ and zero otherwise. Set $\alpha (x) = \Lambda(\mathbf 1_x)$ for all $x\in [0,1]$. \\

Now for $f\in \mathcal{C}[0,1]$, define
$$f_n =\sum_{j=1}^n f(j/n) (\mathbf{1}_{j/n}-\mathbf{1}_{(j-1)/n}).$$

Since $f$ is continuous,  it is uniformly continuous on $[0,1]$ and so $\|f_n-f\|_\infty\to 0$. Thus
$$ \lim_n \Lambda (f_n) =\Lambda (f)=L(f).$$

Using the definition of $\alpha$ we get
$$\Lambda(f_n)=\sum_{j=1}^n f(j/n) (\alpha(j/n)-\alpha((j-1)/n)).$$

This in turn implies 
$$ \Lambda (f)=\lim_n \Lambda (f_n)=\int_0^1 f \,d\alpha.$$
 Now to see that $\| L \| = V(\alpha)$ : \\

 Let $\varepsilon>0$ and choose $f\in \mathcal C[0,1]$ such that $\|f\|_\infty\le 1$ and $\|L\|\le |L(f)|+\varepsilon$, we apply (\ref{original}) and we get

$$\|L\|\ \le|L(f)|+\varepsilon=\left|\int_0^1f(x)d\alpha(x)\right|\ + \varepsilon  \le \alpha(1)-\alpha(0)+\varepsilon=V(\alpha)+\varepsilon. $$

It is possible to normalize $\alpha$ and in this case we easily have the other inequality, that is, 
$$V(\alpha)=\alpha(1)-\alpha(0)=\alpha(1)=\Lambda(\mathbf 1_1)\le \|\Lambda\|=\|L\|.$$

\end{proof}

\textbf{Remarks}. \begin{itemize} 
\item [(1)] The standard textbook's proof uses Hahn-Banach's theorem (\cite{MR0467220},\cite {MR1861991}), but the original proof of  F. Riesz does not use it. See \cite{MR1068530} section 50 and \cite{Riesz1},\cite{MR0051436}.
\
\item[(2)]
 E. Helly \cite{Helly} should have similar results. J. Radon extended theorem \ref{R} to compact subsets $K\subset\mathbb{R}^n $ \cite{MR925205}. S. Banach and S. Saks extended the result to compact metric spaces, see appendix of \cite{MR0167578} and \cite{MR1546062}. The proof by S. Saks is particularly elegant and clean. For compact Hausdorff spaces the theorem was proven by S. Kakutani \cite{MR0005778} and for normal spaces by A. Markoff \cite{Mar}. Nowadays this theorem is also known as Riesz-Markoff or Riesz-Markoff-Kakutani theorem. More information on the history  of this theorem can be found in \cite{MR1681462} p. 231, the references therein,  \cite{MR3408971} p.238 and \cite{MR0753703}.
 \item  [(3)] Positivity of a linear functional $L$ implies continuity of $L$. To see it, we take the function $\textbf{1}(x)=1$  for all $x\in K$, then $\textbf{1}\in\mathcal C(K)$ and $|f(x)|\leq\|f\|_{\infty}\textbf{1}(x)$, therefore $$\|f\|_{\infty}\,\textbf{1}(x)\pm f(x)\geq0\qquad\mbox{implies}\qquad\,\|f\|_{\infty}L(\textbf{1})\pm L(f)\geq0$$ so $|L(f)|\leq L(\textbf{1})\|f\|_{\infty}.$ See \cite{MR1681462}  Prop. 7.1.

\end{itemize}

\section {Main Result} \


 Next theorem is our main result. It is a generalization of Theorem \ref{R} to continuous functions defined on arbitrary compact sets $K\subset \mathbb R$. Since an ordinary Riemann-Stieltjes integral is not defined for functions on
 general compact $K$, we shall introduce the Lebesgue integral which makes sense for such functions. In the Appendix, we collect the  basic facts and definitions of measure theory we need.

\begin{teorema}\label{RC}
Let  $K$ a compact subset of $\mathbb R$ and let $\ell:\mathcal {C}(K)\rightarrow \mathbb R$ be a positive linear functional. Then, there is a unique finite Borel measure $\mu$ such that $\mu(K)=\|\ell \|\ _{C(K)^*}$ and
\begin{equation}\label{eqn}  \ell f=\int_Kfd\mu. \end{equation}
\end{teorema}

\noindent \textit{Proof}. The proof proceeds in stages. \\
\begin{itemize}
\item[i)]\textit{Integral representation}. Let $[a,b]$ be a closed and bounded interval containing $K$. Let $r:\mathcal C[a,b]\longrightarrow \mathcal C(K)$ be the restriction operator, that is, for every $f\in \mathcal C[a,b]$, $r(f)(x)=f(x)$ for $x\in K$. It is clear that $r$ is a bounded linear operator, so we can define its transpose operator, see \cite{MR3408971} p.11, also known as adjoint, see \cite{MR1861991}. Recall $r^t$ is defined as follows $r^t: \mathcal C(K)^* \to \mathcal C[a,b]^*$, $r^t(\ell)(f)=\ell(r(f))$ for $f\in \mathcal C[a,b]$; the expression $\ell(r(f))$ assigns a scalar to each function $f\in \mathcal C[a,b]$. \\
Let $\ell$ be a positive linear functional in $C(K)$ and we define $Lf=r^t(\ell) (f)=\ell (rf)$. Since $\ell$ and $r$ are positive linear functionals, so is $L$ and we can apply theorem \ref{R} and (c) in the Appendix to find a monotone increasing function $\alpha$ and an associated Borel measure $\mu$ such that
\begin{equation}\label{nosabemos}
Lf=r^t(\ell)(f)=\int_a^bfd\alpha=\int_a^bfd\mu
\end{equation}
for every $f\in\mathcal C[a,b]$.\\
 Denote $K^c := [a,b]\setminus K$.
We will show that $\mu(K^c)=0$.  
Let  $\varepsilon>0$ and choose $F_\varepsilon$ as a closed subset of $K^c$ such that 
\begin{equation}\label{Fe}
\mu (K^c\setminus F_\varepsilon)<\varepsilon,
\end{equation}
 see (a) in the Appendix. 

Let $\tilde f\in \mathcal C[a,b]$ be a continuous function such that $\tilde f(x)=1$ if $x\in K$, $\tilde f(x)=0$ if $x\in F_\varepsilon$ and $\|\tilde f\|_\infty\le 1$.  One can take for instance $$ \tilde f(x)=\frac{d(x,F_\varepsilon)}{d(x,F_\varepsilon)+d(x,K)},$$  where $d(x,A)=\inf_{y\in A}|x-y|$.  Note that since $|d(x,A)-d(y,A)|\le|x-y|$ the function $d(x,A)$  is even uniformly continuous, (cf. Urysohn's Lemma. \cite{MR1681462}, 4.15.).  Therefore

\begin{equation*}L(\tilde f)=\int_a^b\tilde fd\mu=\int_Kd\mu+\int_{K^c\backslash F_\varepsilon}\tilde fd\mu
+\int_{F_\varepsilon}\tilde fd\mu
 \end{equation*}
  The third integral on the right is equal zero, by definition of $\tilde f$. We can estimate  the second integral as follows,

\begin{align*}0\leq\int_{K^c\backslash F_\epsilon}\tilde fd\mu\leq\int_{K^c\backslash F_\varepsilon}d\mu=\mu(K^c\backslash F_\varepsilon)<\varepsilon, \end{align*}
since $\tilde f\leq1$  and using  \eqref{Fe}. Then
\begin{align*} L(\tilde f)<\int_Kd\mu+\varepsilon=\mu(K)+\varepsilon.\end{align*}
 We have that
$$\mu(K)+\mu(K^c)=\int_a^bd\mu=L(\mathbf{1}_{[a,b]} )=L(\tilde f)<\mu(K)+\varepsilon,$$

 The third equality  follows from $r(\tilde f)=r(\mathbf{1}_{[a,b]} )$. Thus $0\leq\mu(K^c)<\varepsilon$,
 since $\mu(K)<\infty$.
\\

To conclude, let $f\in\mathcal C(K)$ and  $f^*$ a continuous extension of $f$ to the closed interval $[a,b]$. We
can do this extension taking, for example, straight lines as follows:
since $K^c$ is an open subset of $[a,b]$, it is at most a countable union of pairwise disjoint open intervals $(\alpha_i, \beta_i)$ intersected with the interval  $[a,b]$, (see Lindeloef's thm., \cite{MR0151555} Prop.9. p.40). For $ x\in (\alpha_i,\beta_i)$ we define
$$f^*(x) =(1-t) f(\alpha_i)+t f(\beta_i)$$ if $x= \alpha_i (1-t)+t \beta_i$ for $t\in (0,1).$
 The function $f^*$ is continuous on the interval $[a,b]$ since on $K$ coincides with  the continuous function $f$ and on $K^c$ consists of straight lines, (cf. Tietze's Theorem \cite{MR1681462}, 4.16). \

   Then we have 
\begin{equation} \label{eqn: 2}
\ell(f)=\ell(r(f^*))=Lf^*=\int_a^b f^*d\alpha=\int_a^b f^*d\mu=\int_Kf^*d\mu=\int_Kfd\mu.\end{equation}
as was to be shown.\\

\item[ii)]\textit{Conservation of norm}.  Take $f\in \mathcal{C}(K)$ such that $\|f\|_\infty\le 1$. Since \eqref{eqn} holds we have,
$$|\ell (f)|= \left |\int_Kfd\mu  \right|\le \|f\|_\infty\,\mu (K) \le \mu(K).$$

For the reverse inequality, let $\textbf{1}(x)=1$ for all $x$,  as defined in remark (3), so 
$$\|\ell\|\ge |\ell(\textbf{1})|=\left |\int_K \textbf{1}\, d\mu \right| = \mu(K),$$
we can conclude that $\mu(K) = \|\ell \|$.  \\

\item[iii)]\textit{Uniqueness}. Suppose $\mu$ and $\nu$ are finite measures that satisfy \eqref{eqn}. Since $\mu$ and $\nu$ are regular measures, from (a) in the Appendix, it is enough to show that $\mu(C) =\nu(C)$ for any closed set $C$ of $K$. Let $C$ a nonempty closed set of $K$ and set $f_k (x):= \max\{0,1- k d(x,C)\}$ for all $k$ and $x\in K$, where $d(x,C)=\inf_{y\in C}|x-y|$. These functions are bounded, by $0$ and $1$, and continuous. Thus $f_k$ belongs to $\mathcal C(K)$ for all $k$. Notice that they form a sequence that decreases to the indicator of $C$, i.e., $f_k\downarrow \textbf{1}_C$, where  $\textbf{1}_C(x)=1$ if $x\in C$ and $\textbf{1}_C(x)=0$ if $x\notin C$. Thus, for all $k$ we must have that  $\int_K f_k d\mu =\int_K f_k d\nu$, and so we can use the dominated convergence theorem, see (b) in the Appendix, to conclude that 
$$\mu(C) = \lim_k \int_K f_k d\mu =\lim_k \int_K f_k d\nu= \nu (C). $$
\end{itemize}
 \hfill\qed\\

{\bf{Remarks}} \begin{itemize} 

\item[(a)] It is possible to represent the linear positive functionals acting on $\mathcal C(K)$ as Riemann-Stieltjes integrals, similar to the original work of F. Riesz. This follows immediately from the chain of equalities \eqref{eqn: 2}. The caveat is that we cannot use $f$ directly in order to define the Riemann-Stieltjes integral, but any continuous extension of $f$ works, cf. theorem \ref{iso1} below. This integral is independent of the  extension of $f$.

\item[(b)] As just seen, the use of compact set $K$ above allows us to extend the continuous functions  to the entire interval $[a,b]$, using an elementary version of the Tietze's theorem. This construction is in general not possible if $K$ is an arbitrary subset of the real line.


\end{itemize}

\subsection{Isomorphic spaces}\

As a consequence of the previous results, we shall see that two spaces of functionals are practically the same.
One of the spaces consists of  Lebesgue integrals on compact subsets of $[a,b]$ and the other of Riemann-Stieltjes integrals over the whole interval $[a,b]$. In this way we show how the  Lebesgue integral representation, that was introduced to represent functionals in the case of general compact sets, can be seen as a Riemann-Stieltjes integral.
To state this precisely we need to introduce the terms {\it isomorphic} and 
 {\it constant in  $K^c$}.
 
 A transformation $T$ which preserves the norm, that is $ \| Tx\| =\|x\|$, is called an {\it isometry.}
 Two normed vector  spaces $X$ and $Y$ are said to be {\it isomorphic} if there is a linear, bijective, isometry
   $T: X\to Y$. Such functions are  called  {\it isomorphisms}.
 Since an isomorphism preserves the linear as well as the metric structure of the spaces, two isomorphic spaces can be considered identical, the isomorphism corresponding  just to a labeling of the elements.  We say that the monotone function $\alpha  \text{ \it is constant in } K^c$  if it is constant in each interval of $K^c$.


 Recall that  $\mathcal C(X)^*_+$ denotes the set of positive linear functionals on 
$\mathcal C(X)$.\\

The result mentioned above can be then stated as follows:

 \begin{teorema}\label{iso1}
  The normed spaces $ \left \{  L_\alpha\in \mathcal C[a,b]^*_+ : \alpha  \text{ is constant in } K^c    \right \}$   and $ \mathcal C(K)^*_+$ are isomorphic.
 \end{teorema}
  Before we prove this theorem we need two preparatory results.
 
\begin{proposicion} \label{iso}
 $r^t :\mathcal C(K)^*_+ \to \mathcal C[a,b]^*_+$  is an isometry.
 \begin{proof}
 
 $\|r^t \ell\|_{\mathcal C[a,b]^*_+}=V(\alpha)=\alpha (b)-\alpha(a)=\mu([a,b])=\mu(K) +\mu([a,b]\setminus K)=\mu(K)
 =\|\ell\|_{\mathcal C(K)^*_+}$.
 
 The first equality follows from Theorem \ref{R}. The function $\alpha$ depends on $\ell$. The second is the definition of the total variation of $\alpha$ and the third is the definition of $\mu$. The last two equalities follow from the construction of Theorem \ref{RC}.

%
 
 \end{proof}
\end{proposicion}
 We denote the range of $r^t$ by Rang $r^t =\left \{ L\in  \mathcal C[a,b]^*_+: \exists l    \in  \mathcal C(K)^*_+  \text{ s.t. } L=r^t l  \right \}$
  \begin{proposicion}\label{Rang4}
 Let $ L_\alpha$ denote the functional with corresponding monotone function $\alpha$ as introduced in \eqref{original}. Then
 $$\text{Rang } r^t=\left \{  L_\alpha\in \mathcal C[a,b]^*_+ : \alpha  \text{ is constant in } K^c    \right \}$$
 \begin{proof}
 "$\subset$"
 
 Let $L\in\text{Rang } r^t\subset \mathcal C[a,b]^*_+$. Then there exists $\ell\in \mathcal C(K)^*_+$ such that as in \eqref{nosabemos}
 $$r^t(\ell)(f)=Lf=L_\alpha f=\int_a^bfd\alpha=\int_a^bfd\mu$$
  As was shown in the proof of Theorem \ref{RC} i),   $\mu(K^c) =0$.
   Since $K^c$ is a countable union of intervals, these have $\mu$ measure  zero.  By the relation which is given in \eqref{cara} below, between the measure $\mu$ and the monotone function $\alpha$  we conclude that $\alpha$ 
 is constant in each one of the intervals of $K^c$.
  
  " $\supset$   "

Let $ L_\alpha\in \mathcal C[a,b]^*_+$ with $\alpha$ constant in each interval of $K^c$ and  $\mu$ be the measure associated with this $\alpha$, as in Appendix (c). Define $\ell \in \mathcal C(K)^*_+$ as $\ell h=\int_Kh d\mu$. We shall show that
$r^t(\ell)(f)=L_\alpha f$ for every $f\in \mathcal C[a,b]$. Since $\alpha$ constant in each interval of $K^c$ this implies, using again \eqref{cara}, that $\mu(K^c) =0$. Then we have
$$   L_\alpha f= \int_a^bfd\alpha=\int_a^bfd\mu = \int_Kfd\mu=\int_Kr(f)d\mu= \ell(r(f))= r^t(\ell)(f)             $$
where $r(f)$ denotes, as  in Theorem \eqref{RC} i) above, the restriction of $f$ to $K$.
 \end{proof}
 \end{proposicion}
   \begin{proof} [Proof of Theorem \ref{iso1}].\
   
 From  Proposition \ref{iso} and Proposition \ref{Rang4} it follows that $r^t$ is a bijective isometry. Since 
 $r^t$ is linear as follows from its definition, then it is an isomorphism.
 \end{proof}

{\bf Acknowledgments} 
We thank C. Bosch and Ma. C. Arrillaga for useful comments. We are grateful to Ma. R. Sanchez  for her help in the search of 
bibliographical information.

\section{Appendix}\

A collection of subsets $\mathcal A$ of $X$ is called an \textit{$\sigma$-algebra} if it is closed under  finite (countable) union, complements and $X\in \mathcal A$. If our space is $\mathbb R$, the \textit{Borel $\sigma$-algebra}, $\mathcal B_{\mathbb R}$, is the smallest $\sigma$-algebra containing all the open intervals.  A function $\mu: \mathcal A\to [0,\infty]$, where $\mathcal A$ is a $\sigma$-algebra, it is called a \textit{measure} if it  is countable additive, that is $\mu(\bigcup A_n) =\sum \mu (A_n)$ whenever $\{A_n\}$ is a disjoint sequence of elements in $\mathcal A$, and $\mu (\varnothing) =0$. A \textit{Borel measure} is a measure defined on $\mathcal B_{\mathbb R}$. We say that a measure is \textit{regular} if every measurable set can be approximated from above by open measurable sets and from below by compact measurable sets. A function $f$ from $(X,\mathcal A, \mu)$ to $(\mathbb R, B_{\mathbb R} )$ is $\mathcal A$-measurable if $\{x: f(x)\le t\}\in \mathcal A$ for all $t\in \mathbb R$.\\

The following results are used in the proof of Theorem \ref{RC}.
\begin{itemize}
   \item[(a)] Every Borel measure in a metric space is regular. We will only use inner regularity, that is, for every Borel set $A$ and every $\varepsilon>0$ there exist a compact set $F_\varepsilon$ such that $F_\varepsilon \subset A$ and $\mu (A\setminus F_\epsilon)<\epsilon$ . \cite{MR2267655} Thm 7.1.7. or \cite{Co} Lemma 1.5.7.
\item[(b)] (Dominated convergence theorem) Let $(X,\mathcal A, \mu)$ a measure spaces. Let $g$ be a $[0,\infty]$-valued integrable function on $X$, that is, $\int g d\mu <\infty$, and let $f, f_1, f_2, \ldots$ real-valued $\mathcal A$-measurable functions on $X$ such that $f(x)= \lim_n f_n (x)$ and $|f_n(x)|\le g(x)$. Then $f$ and $\{f_n\}$ are integrable and $\int f d\mu = \lim_n \int f_n d\mu$.
  \item [(c)] Given a normalized monotone function $\alpha$  in the closed interval $[a,b]$, there is a unique Borel measure $\mu$ associated with it. This can be seen as follows (see for example  \cite{MR1253752}): for $a\le s \le t \le b\,$ let define $\, \langle s, t]$ where

\end{itemize}

Let $$\mathcal F_0 = \bigg\{\bigcup_{\text{ finite}}\langle s_k, t_k]: \langle s_k, t_k] \subset [a,b] \text{ pairwise disjoint}\bigg\}$$

Then $\mathcal F_0$ is an algebra of subsets of $[a,b]$ and  therefore we can define a set function as 
\begin{equation}\label{cara}
\mu_0 \left(\bigcup_{\text{finite}}\langle s_k, t_k] \right)= \sum_{\text{finite}} \alpha(t_k)-\alpha(s_k). 
\end{equation}
Moreover, $\mu_0$ has a unique extension to a measure in the smallest $\sigma$-algebra containing $\mathcal F_0$ (Caratheodory's Theorem). See \cite {MR0200398} .  Moreover, for any continuous function $f$ it happens that
\begin{equation}\label{measure}
\int_a^bfd\alpha=\int_a^bfd\mu
\end{equation}
where the integral on the left is a Riemann-Stieltjes integral, whereas the integral on the right is an integral in the sense of Lebesgue.


\begin{thebibliography}{20}
\bibitem {MR0200398}  Bartle, Robert G.  The elements of integration.
John Wiley \& Sons, Inc., New York-London-Sydney  1966 {\rm x}+129 pp.
  
\bibitem{MR2267655}
V.~I. Bogachev, \emph{Measure theory. {V}ol. {I}, {II}}, Springer-Verlag,
  Berlin, 2007.
 \bibitem{Co}
D. Cohn,  \emph{Measure theory}, secon ed., Birkh\"auser, Boston, Mass. 2013.


 \bibitem {MR1253752}  Doob, J. L.  Measure theory.
Graduate Texts in Mathematics, 143. Springer-Verlag, New York,  1994. {\rm xii}+210 pp. ISBN: 0-387-94055-3 
\bibitem{MR1681462}
Gerald~B. Folland, \emph{Real analysis}, second ed., Pure and Applied
  Mathematics (New York), John Wiley \& Sons, Inc., New York, 1999, Modern
  techniques and their applications, A Wiley-Interscience Publication.
  
\bibitem{MR0753703}  Gray, J. D.  The shaping of the Riesz representation theorem: a chapter in the history of analysis.
 Arch. Hist. Exact Sci.  31  (1984),  no. 2, 127--187.



\bibitem{Helly}
E.~Helly, \emph{\"Uber lineare Funktionaloperationen}, Wien Ber. \textbf{121}
  (1912), 265--297.

\bibitem{MR0005778}
Shizuo Kakutani, \emph{Concrete representation of abstract {$(M)$}-spaces. ({A}
  characterization of the space of continuous functions.)}, Ann. of Math. (2)
  \textbf{42} (1941), 994--1024. 

\bibitem{MR0467220}
Erwin Kreyszig, \emph{Introductory functional analysis with applications}, John
  Wiley \& Sons, New York-London-Sydney, 1978. 

\bibitem{Mar}
A.~Markoff, \emph{On mean values and exterior densities}, Mat. Sbornik
  \textbf{4 (46)} (1938), no.~1, 165--191.

\bibitem{MR925205}
Johann Radon, \emph{Gesammelte {A}bhandlungen. {B}and 1}, Verlag der
  \"Osterreichischen Akademie der Wissenschaften, Vienna; Birkh\"auser Verlag,
  Basel, 1987, With a foreword by Otto Hittmair, Edited and with a preface by
  Peter Manfred Gruber, Edmund Hlawka, Wilfried N{\"o}bauer and Leopold
  Schmetterer. 

\bibitem{MR751959}
Michael Reed and Barry Simon, \emph{Methods of modern mathematical physics.
  {I}}, second ed., Academic Press Inc. [Harcourt Brace Jovanovich Publishers],
  New York, 1980, Functional analysis. 

\bibitem{Riesz}
F.~Riesz, \emph{Sur les op\'erations fonctionnelles lin\'eares}, Comptes Rendus
  Acad. Sci. Paris \textbf{149} (1909), 974--977.

\bibitem{Riesz1}
 \emph{Demonstration nouvelle d'un th\'eor\`eme concernant les op\'erations},
  Annales Ecole Norm. Sup. \textbf{31} (1914), 9--14.

\bibitem{MR0051436}
Fr{\'e}d{\'e}ric Riesz, \emph{Sur la repr\'esentation des op\'erations
  fonctionnelles lin\'eaires par des int\'egrales de {S}tieltjes}, Comm. S\'em.
  Math. Univ. Lund [Medd. Lunds Univ. Mat. Sem.] \textbf{1952} (1952), no.~Tome
  Supplementaire, 181--185. 

\bibitem{MR1068530}
Frigyes Riesz and B{\'e}la Sz.-Nagy, \emph{Functional analysis}, Dover Books on
  Advanced Mathematics, Dover Publications, Inc., New York, 1990, Translated
  from the second French edition by Leo F. Boron, Reprint of the 1955 original.
  

\bibitem{MR0151555}
H.~L. Royden, \emph{Real analysis}, The Macmillan Co., New York;
  Collier-Macmillan Ltd., London, 1963. 

\bibitem{Rud}Walter Rudin, \emph{Real and Complex Analysis} 3rd. ed., 
McGraw-Hill, Inc., 1987, New York, NY, USA. 

\bibitem{MR1546062}
S.~Saks, \emph{Integration in abstract metric spaces}, Duke Math. J. \textbf{4}
  (1938), no.~2, 408--411. 

\bibitem{MR0167578}
Stanis{\l}aw Saks, \emph{Theory of the integral}, Second revised edition.
  English translation by L. C. Young. With two additional notes by Stefan
  Banach, Dover Publications, Inc., New York, 1964. 

\bibitem{MR1861991}
Martin Schechter, \emph{Principles of functional analysis}, second ed.,
  Graduate Studies in Mathematics, vol.~36, American Mathematical Society,
  Providence, RI, 2002. 

\bibitem{MR3408971}
Barry Simon, \emph{Real analysis}, A Comprehensive Course in Analysis, Part 1,
  American Mathematical Society, Providence, RI, 2015, With a 68 page companion
  booklet. 


\end{thebibliography}
\end{document}